\documentclass[12pt]{article}
\usepackage{amssymb,amsmath,amsthm}
\usepackage{amsmath,amssymb,amsthm,amscd,url}
\usepackage{amsmath,amsfonts,amssymb,amscd,amsthm,epsfig}
\newtheorem{prop}{PROPOSITION}[section]
\newtheorem{lemma}{LEMMA}[section]

\newtheorem{theorem}{THEOREM}[section]

\usepackage{graphicx}
\usepackage{epstopdf}
\theoremstyle{assumptions}
\newtheorem{assmpn}{Assumption}[section]
\DeclareMathOperator{\Var}{Var}
\DeclareMathOperator{\Cov}{Cov}
\allowdisplaybreaks[1]

\begin{document}

\numberwithin{equation}{section}
\title{Current fluctuations for  independent random walks in multiple dimensions}
 
\author{Rohini Kumar\\
 Statistics and Applied Probability\\
 UCSB\\
 Santa Barbara, CA-93106\\
 {\tt kumar@pstat.ucsb.edu}}
 
\maketitle

 \begin{abstract}
Consider a system of particles evolving as independent and identically distributed (i.i.d.)\ random walks. 
Initial fluctuations in the particle density get translated over time with velocity $\vec{v}$, the common mean velocity of the random walks. 
Consider a box centered around an observer who starts at the origin and moves with constant velocity $\vec{v}$.  To observe interesting fluctuations beyond the translation of initial density fluctuations,  we measure  the net flux of particles over time into this moving box.  We call this the ``box-current" process. 
 
 We generalize this current process to a distribution valued process. Scaling time by $n$ and space by $\sqrt{n}$ gives current fluctuations of order $n^{d/4}$ where $d$ is the space dimension. The scaling limit of the normalized current process is a distribution valued Gaussian process with given covariance. The limiting current process is equal in distribution to the solution of a given stochastic partial differential equation which is related to the generalized Ornstein-Uhlenbeck process. 
  \end{abstract}

\bigskip
 


\allowdisplaybreaks[1]

\def\v{\vec{v}}

\def\sk{\sum_{k=0}^{[Tn^\beta]}}
\def\Zd{\mathbb Z^d}
\def\Xt{\tfrac{X_{m,j}(nt)-[n\vec{v}t]}{\sqrt{n}}} 
\def\Xtj{\tfrac{X_{m,1}(nt)-[n\vec{v}t]}{\sqrt{n}}} 
\def\Xs{\tfrac{X_{m,j}(ns)-[n\vec{v}s]}{\sqrt{n}}} 
\def\Xsj{\tfrac{X_{m,1}(ns)-[n\vec{v}s]}{\sqrt{n}}} 
\def\sm{\sum_{m\in\Zd}}
\def\seta{\sum_{j=1}^{\eta_0(m)}}
\def\pst{\phi(\Xt)-\phi(\Xs)}
\def\pstj{\phi(\Xtj)-\phi(\Xsj)}
\def\sB{{\bf 1}\{X_{m,j}(ns)\in B_{M\sqrt{n}}+[n\vec{v}s]\}}
\def\tB{{\bf 1}\{X_{m,j}(nt)\in B_{M\sqrt{n}}+[n\vec{v}t]\}}
\def\sBm{{\bf 1}\{X(ns)\in B_{M\sqrt{n}}+[n\vec{v}s]-m\}}
\def\tBm{{\bf 1}\{X(nt)\in B_{M\sqrt{n}}+[n\vec{v}t]-m\}}
\def\Xtm{\tfrac{X(nt)-[n\vec{v}t]+m}{\sqrt{n}}}
\def\Xsm{\tfrac{X(ns)-[n\vec{v}s]+m}{\sqrt{n}}}
\def\pstm{\phi(\Xtm)-\phi(\Xsm)}
\def\Xts{\tfrac{X(nt)-X(ns)-[n\vec{v}t]+[n\vec{v}s]}{\sqrt{n}}}
\def\l{\left}
\def\r{\right}
\def\Bl{\Biggl}
\def\Br{\Biggr}
\def\bl{\biggl}
\def\br{\biggr}
\def\supk{\sup_{kn^{-\beta}\leq t< (k+1)n^{-\beta}}}
\def\nb{n^{1-\beta}}
\def\sma{\sum_{m\in B_{n^{1/2+\alpha}}}}
\def\ptk{\phi(\Xt)-\phi(\tfrac{X(k\nb)-[k\nb v]}{\sqrt{n}})}
\def\nkb{kn^{1-\beta}}
\def\nk1{(k+1)n^{1-\beta}}
\def\k{{\bf  \kappa}}
\def\a{{\bf a}}
\def\S{ \mathcal S(\mathbb R^d)}
\def\Sdual{ \mathcal S'(\mathbb R^d)}
\def\B{\k B}

\def\smG{\sum_{m\in B_{n^{1/2+\alpha}}}\seta {\bf 1}\l\{G_{m,j}^{[a,b)}\r\}}

\def\wk{\tilde{w}_{\xi_n}(J_{k,n})}

\section{Introduction}\label{sec:intro}
In a system of independent random walks in dimension $d\geq 1$, the hydrodynamic limit of particle density, under Euler scaling, satisfies the scalar  conservation law
\begin{equation}
\label{conservation_law}
 \partial_t u(x,t)+ \nabla_x\cdot f (u(x,t))=0, \end{equation} with  $u(x,0)=u_0(x)$ as the initial particle density and where $t\geq 0$, $x\in\mathbb R^d$ and $f(u)=u\v$, $\v$ being the mean drift of the random walks. %Thus, over time, on the macroscopic scale we see a shift of the initial profile with velocity v. 
 What we mean by this hydrodynamic limit is the following. If we start with a sequence of initial particle configurations $\{\eta_0^n(m):m\in\Zd\}$ satisfying \[\frac{1}{n^d}\sum_{x\in nA\cap \Zd}\eta_0^n(x)\longrightarrow \int_A u_0(x)dx \text{  in probability}\] where $A\subset \mathbb R^d$  is a bounded box,
 then  the particle configurations at subsequent times, $\eta_t^n(\cdot)$, satisfy 
 \[\frac{1}{n^d}\sum_{x\in nA\cap \Zd}\eta_{nt}^n(x)\longrightarrow \int_A u(x,t)dx \text{  in probability} \]where $u(x,t)$ satisfies \eqref{conservation_law} with initial condition $u(x,0)=u_0(x)$. 
 This hydrodynamic limit picture indicates that on time and space  scales of $O(n)$ we see shifts in the particle profile along characteristic lines of \eqref{conservation_law} 
given by $ny+n\vec{v}t$, $y\in\mathbb R^d$. To capture diffusive fluctuations about these
characteristic lines, we consider a box of radius $O(\sqrt{n})$ moving with the characteristic velocity $\v$, and we look at the net inward  flux of particles across the
boundary of this moving box.  The measurement of this net flux of particles gives us what we call the `box-current' process.  We generalize this box-current  process to  a distribution-valued process.  In section \ref{sec_current}, we  give a precise  definition of these current processes  and explain the connection between the current process and  fluctuations about the hydrodynamic limit.

The question we resolve in this paper is the size and the distribution of the scaling limit of current fluctuations. It turns out that the current process scaled by $n^{-d/4}$, where $n$ is the scaling parameter, gives us Gaussian scaling limits.
  The limiting distribution-valued current process can also be attained as the solution of a stochastic partial differential equation.    An alternate representation of the limiting current process as the solution of a stochastic differential equation is given in section  \ref{sec:spde}. The limiting current process is expressed as the sum of two independent  stochastic integrals, where one of the integrals is a generalized Ornstein-Uhlenbeck process.
     
  In 1985, current fluctuations in a system of independently moving particles was studied in \cite{Durr} in connection with the asymptotic behavior of trajectories in
an infinite particle system with collisions. Sepp\"al\"ainen, in \cite{Timo-IRW}, found the scaling limit of current across characteristics in the one dimensional system of asymmetric independent random walks. His work was extended in \cite{kumar} where shifts in characteristics at the diffusive scale were allowed. In both cases current fluctuations were of order $n^{1/4} $, where $n$ was the scaling factor, and Gaussian scaling limits were obtained. In \cite{kumar}, a representation of the limiting current process as the sum of two stochastic integrals was given for the one dimensional case which agrees with the results in section \ref{sec:spde} of this paper. While in the one dimensional system of independent random walks there is a connection between the position of a tagged particle under elastic collisions and a current process (see \cite{Durr}),  there is no such obvious connection in higher dimensions.

Equilibrium fluctuations in particle density, in a system of independent Markovian particles in multiple dimensions, was studied  by Anders Martin-L\"of in 1976 \cite{Martin-Lof}.  Our results about current fluctuations hold in the equilibrium as well as non-equilibrium case, under the assumption that the initial distribution of number of particles at each site on the lattice $\mathbb Z^d$ is i.i.d.  It is interesting to note that in the absence of fluctuations in the initial particle configuration (i.e. if we start with a deterministic initial configuration), our limiting current process is a generalized Ornstein-Uhlenbeck process.  Generalized Ornstein-Uhlenbeck processes appear as the fluctuation limits of several infinite particle systems in the literature, for example see chapter 11 of Kipnis and Landim  \cite{KL}, \cite{BG86}, \cite{HS78} and \cite{Martin-Lof}. %In references \cite{KL} and \cite{Martin-Lof}, the generalized Ornstein-Uhlenbeck process  is the limiting distribution of equilibrium fluctuations about the hydrodynamic limit.

Fluctuation results for other asymmetric models in multiple dimensions can be found in the literature.  We mention some of these. In multiple dimensions, equilibrium fluctuations for a tagged particle in asymmetric exclusion processes  and  in asymmetric zero range processes have been studied in \cite{Sethu_Vara_Yau, Sethu-ASEP} and  \cite{Sethu-ZRP} respectively.  The diffusion coefficient for the two dimensional asymmetric simple exclusion process was studied in \cite{Yau}.  In \cite{Sethu-1}, Sethuraman proves superdiffusivity  of  occupation-time variance in the two dimensional asymmetric exclusion process. 

   In this paper we extend the results of \cite{Timo-IRW} and \cite{kumar} to multiple dimensions. We also generalize their results from a simple real-valued current process to a distribution-valued current process. Our proof  parallels  the proof in \cite{Timo-IRW} and \cite{kumar} in several places.  Some significant differences  arise in the treatment of  a distribution-valued  process. As is typical for process-level distributional limits,  our proof is in two stages: first, we prove convergence of finite dimensional distributions and second, we prove tightness of the current process.  We omit parts of the proof that are similar to parts in \cite{kumar} and refer the reader to \cite{kumar} for details.
   
    This paper is structured as follows. Section \ref{sec:model} contains a  description of the independent random walks model, the current process %and our process, some technical assumptions 
    and a  statement of our results. The  stochastic partial differential equation satisfied by the limiting current process and its solution are also mentioned here. In section \ref{sec:finite-dim} we find the finite dimensional distribution limits of the process. We prove tightness of the process and complete the proof in section \ref{sec:tightness}.

    %The last section sketches the proof of how  the stochastic partial differential equation satisfied by the limiting current process can be obtained.

\section{Model and Results}
\label{sec:model}
\subsection{Independent random walks model}
Let $d\geq 1$ denote the space dimension.  %The independent random walk model in dimension $d$ consists of particles distributed on the integer lattice $\mathbb Z^d$  that evolve as i.i.d.\ random walks.  
We start with an initial configuration of particles on  $\mathbb Z^d$ denoted by $\{\eta_0(x):x\in\Zd\}$, i.e. $\eta_0(x)$ is the number of particles at site $x$ at time $0$. These particles evolve like independent, identically distributed (i.i.d.) continuous time random walks. Let $X_{m,j}(t)$, $m\in\Zd, j=1,\hdots,\eta_0(m)$ denote the position at time $t$ of the $j$-th random walk starting at site $m$.  Let  $\{\eta_t(x):x\in\Zd\}$ denote the occupation variables at time $t$ i.e. $\eta_t(x)$ is the number of random walks (particles) at site $x$, at time $t$. %For notational convenience we will drop the superscript $n$ in $X_{m,j}^n$ and simply write $X_{m,j}$. 
 Clearly,
\begin{align*}\eta_t(x):=\sm\seta {\bf 1}\{X_{m,j}(t)=x\},\end{align*}$x\in\Zd$.

The common jump rates of the random walks are given by the probability kernel $\{p(x): x\in \Zd\}$ where we assume \[\sum_{x\in \Zd}p(x)=1.\] The common transition probabilities of the random walks is given by 
\[P(X_{m,j}(s+t)=y|X_{m,j}(s)=x)=\sum_{k=0}^\infty \frac{e^{-t}t^k}{k!}p^{(k)}(y-x),\]
where \[p^{(k)}(z)=\sum_{x_1+x_2+\cdots+x_k=z}p(x_1)\cdotp(x_2)\cdots p(x_k)\] is the $k-$fold convolution of the kernel $p(x)$.

We make the following assumptions.
\begin{assmpn}\label{configuration}
The initial occupation variables $\{\eta_0(m):m\in\Zd\}$ are i.i.d. random variables with exponential moments i.e. \[Ee^{\theta\eta_0(m)}<\infty\]
for some $\theta>0$. Let $\rho_0:=E\eta_0(m)$ and $v_0:=\Var\eta_0(m)$.
\end{assmpn}

\begin{assmpn}\label{rw-assmpn}
Given $\{\eta_0(m):m\in\Zd\}$, the evolution of the random walks $\{X_{m,j}(t)-X_{m,j}(0):m\in\Zd, j=1,\hdots , \eta_0(m)\}$ are i.i.d. continuous time random walks on $\mathbb Z^d$ starting at the origin,  independent of $\eta_0(\cdot)$. We can assume without loss of generality that  the common jump probability kernel $\{p(x):x\in\Zd\}$ is not supported on a hyperplane in $\mathbb R^d$. Let \[\v:=\sum_xxp(x)\] denote average velocity of the random walks. 
Define a $d\times d$ matrix  ${\bf a}$, related to the second moments of the random walks,  as 
 \[{\bf a}_{i,j}:=\sum_{x=(x_1,\hdots, x_d)\in \Zd}x_ix_jp(x),\] and  choose matrix $\k$ such that \[\k\k^T= {\bf a}.\]
\end{assmpn}

\begin{assmpn}\label{rw-moment}
\[\sum_{x\in\Zd }e^{\delta |x|}p(x)<\infty\]
for some $\delta>0$.
\end{assmpn} 

\subsection{Some standard notation}
We define the norm of   $x=(x_1,\hdots,x_d)\in\mathbb R^d$ as \[|x|:=\max_{1\leq i\leq d}|x_i|.\]
For $x\in\mathbb R$, define
\begin{equation}
\label{integer-part}
[x]:=\begin{cases}
\text{largest integer }\leq x, &\text{ if }x\geq 0\\
\text{smallest integer }\geq x, &\text{if }x\leq 0.
 \end{cases}
 \end{equation}
 Define
 \[[x]:=([x_1],\hdots,[x_d])\]  for  $x=(x_1,x_2,\hdots, x_d)\in\mathbb R^d$.\\
Let 
 \[\partial^\alpha:=\l(\frac{\partial}{\partial x_1}\r)^{\alpha_1}\cdots \l(\frac{\partial}{\partial x_d}\r)^{\alpha_d}\]  for a multi-index  $\alpha=(\alpha_1,\hdots ,\alpha_d)$ with $\alpha_i\geq 0$.
For any nonnegative integer $N$ and multi-index $\alpha=(\alpha_1,\hdots ,\alpha_d)$,  $\alpha_i\geq 0$, define
 \[||f||_{(N,\alpha)}:=\sup_{x\in\mathbb R^d}\l\{(1+|x|)^N|\partial^\alpha f(x)|\r\}.\] 
 Let $\mathcal S(\mathbb R^d)$ denote the space of Schwartz functions. 
 Recall that this space of functions is defined as 
 \[\mathcal S(\mathbb R^d):=\{f\in C^\infty(\mathbb R^d):||f||_{(N,\alpha)}<\infty\  \forall N,  \alpha \}.\]
We denote the dual of the Schwartz space as $\Sdual$. 
For any bounded subset $B$ of $\mathcal S(\mathbb R^d)$, let 
\[q_B(f):=\sup_{\phi\in B}|f(\phi)|,\ \ \ \ \ \ \ f\in\Sdual.\] Then $\{q_B\}$ is a family of semi-norms on $\Sdual$ which defines the strong topology on $\Sdual$. 
 Fix $T>0$, let $D([0,T],\Sdual)$ be the space of mappings from $[0,T]$ to $\Sdual$ that are right continuous and have left-hand limits in the strong topology of $\Sdual$.

\subsection{The current process and fluctuations about the hydrodynamic limit}\label{sec_current}
Let \[B_M:=\{x\in\mathbb R^d:|x|\leq M\}\] be a box centered at the origin with radius $M$.  Suppose this box is centered about an observer who starts at the origin and moves with constant velocity $\v$. 
We define the `box-current process' as the net inward flux of particles across the boundary of this moving box, over time. Thus the  box-current at time $t$ is simply the difference between the number of particles inside the box at time $t$ and the number of particles initially inside the box. %scaled by $n^{-d/4}$. 
Scaling space by $\sqrt{n}$ and time by $n$, we construct a sequence of scaled box-current processes given by 
 \begin{align*}\xi_n(t,{\bf 1}_{B_M}):&=n^{-d/4}\sum_{i\in  B_{M\sqrt{n}}}\left\{\eta_{nt}(i+[n\vec{v}t])-\eta_0(i)\right\}\\
&=n^{-d/4}\sm\seta\l[{\bf 1}\{X_{m,j}(nt)\in B_{M\sqrt{n}}+[n\vec{v}t]\}-{\bf 1}\{m\in B_{M\sqrt{n}}\}\r].\end{align*}
 In  words, $\xi_n(t,{\bf 1}_{B_M})$ measures the cumulative  net current of particles into the box $B_{M\sqrt{n}}$ up to time $nt$ as the box moves with fixed velocity $\v$.

We generalize the above box-current process to a distribution valued process by taking the  weighted average of the differences $\{\eta_{nt}(i+[n\vec{v}t])-\eta_0(i)\}$, where  the weights are given by values of a Schwartz funcion. For $\phi\in \mathcal S(\mathbb R^d)$ and $t\in[0,T]$, %[or $\phi:={\bf 1}_{B_M}$ where $B_M$ is a d dimensional box of radius $M$ centered at the origin] 
define
\begin{equation}
\label{current process}\xi_n(t,\phi):=n^{-d/4}\sum_{i\in \mathbb Z^d}\left\{\phi\left(\tfrac{i}{\sqrt{n}}\right)\left(\eta_{nt}(i+[n\vec{v}t])-\eta_0(i)\right)\right\}.\end{equation}
This is our current process. Observe that $\xi_n(\cdot,\cdot)\in D([0,T],\Sdual)$.
We can rewrite $\xi_n(t,\phi)$ as 
\begin{equation}
\label{current process-2}\begin{split}&\xi_n(t,\phi)\\
&= n^{-d/4}\sum_{i\in \mathbb Z^d}\phi\left(\tfrac{i}{\sqrt{n}}\right)\left\{\sm\seta \l[{\bf 1}\{X_{m,j}(nt)=i+[n\vec{v}t]\}-{\bf 1}\{m=i\}\r]\right\}\\
&=n^{-d/4}\sum_{m\in\Zd}\sum_{j=1}^{\eta_0(m)}\left[\phi\left(\Xt\right)\l(\sum_{i\in \mathbb Z^d}{\bf 1}\{X_{m,j}(nt)=i+[n\vec{v}t]\}\r)-\phi\left(\tfrac{m}{\sqrt{n}}\right)\right]\\
&=n^{-d/4}\sum_{m\in\Zd}\sum_{j=1}^{\eta_0(m)}\left[\phi\left(\Xt\right)-\phi\left(\tfrac{m}{\sqrt{n}}\right)\right].\end{split}\end{equation}
In the last line of \eqref{current process-2} we are summing over all  random walks and adding their contributions to the current. This is a more tractable form of the current  process compared to its initial  definition in \eqref{current process} and we will use this definition henceforth for the current process.

As in \cite{Timo-IRW}, here too there is a connection between the current process and what the author in \cite{Timo-IRW} termed the  ``second-order fluctuations" from the hydrodynamic limit. We explore this connection below. Let $(\Omega, \mathcal F, \mathcal P)$ be the underlying probability measure space  for our  independent random walk process. By our assumptions, under $\mathcal P$, the initial occupation variables $\{\eta_0(m),m\in\Zd\}$ are i.i.d. random variables and are independent of  the evolution of the random walks $\{X_{m,j}(\cdot)-X_{m,j}(0): m\in \Zd, j=1,\hdots \eta_0(m)\}$. Under these assumptions we get
\begin{equation}
n^{-d/2}\sum_{i\in\Zd}\phi\l(\frac{i}{\sqrt{n}}\r)\eta_{nt}(i+[n\v t])\xrightarrow[]{n\to\infty}\rho_0\int_{x\in\mathbb R^d}\phi(x)dx \quad \text{in }L^2(\mathcal P) \text{ for all }t\geq 0,
\end{equation} 
where $E[\eta_0(m)]=\rho_0$. (For general initial conditions with $E[\eta_0^n(x)]=\rho_0(\frac{x}{\sqrt{n}})$, the right hand side would be $\int\rho(x,t)\phi(x)dx$ where $\rho(x,t)$ is the solution of a diffusion equation with initial condition $\rho_0(\cdot)$. Since $\rho_0(x)\equiv \rho_0$ is a constant in our problem, we get $\rho(x,t)\equiv \rho_0$.)

We can express $\xi_n(t,\phi)$ as the difference of two terms:
\begin{align*}
\xi_n(t,\phi)%&= n^{-d/4}\sum_{i\in \mathbb Z^d}\left\{\phi\left(\tfrac{i}{\sqrt{n}}\right)\left(\eta_{nt}(i+[n\vec{v}t])-\eta_0(i)\right)\right\}\\
&=n^{d/4}\biggl[\l(n^{-d/2}\sum_{i\in \mathbb Z^d}\phi\left(\tfrac{i}{\sqrt{n}}\right)\eta_{nt}(i+[n\vec{v}t])-\rho_0\int_{x\in\mathbb R^d}\phi(x)dx\r)\\
&\qquad \qquad -\l(n^{-d/2}\sum_{i\in \mathbb Z^d}\phi\left(\tfrac{i}{\sqrt{n}}\right)\eta_0(i)-\rho_0\int_{x\in\mathbb R^d}\phi(x)dx\r)\biggr].
\end{align*}
Using the observer as a frame of reference, the first term is the fluctuation from the  hydrodynamic limit  at time $t$ and the second term is the fluctuation from the hydrodynamic limit at time $0$. 
Thus,  $\xi_n(t,\phi)$ is obtained by subtracting the initial fluctuations  from the hydrodynamic limit  from the fluctuations at time $t$ from the hydrodynamic limit, as seen by the observer, and scaling this difference by $n^{d/4}$.  The current process  therefore looks at fluctuations beyond the rigid translation of fluctuations in the initial configuration along the characteristic line $y=x+\v t$. \\

\subsection{Results}
Recall that $\{p(x), x\in\Zd\}$ denote the common jump probabilities of the random walks. Let $X(t)$ be a continuous time random walk starting at the origin, with jump probabilities $p(x)$. Henceforth, in this paper, $X(t)$ will always denote such a random walk.
By the martingale central limit theorem (refer pages 339-340 in \cite{Ethier-Kurtz}), we have \[\frac{X(nt)-[n\vec{v}t]}{\sqrt{n}}\Longrightarrow \k B(t)\]  where  $B(t)$ is standard d-dimensional Brownian motion. 
The probability density function of $\k B(t)$ is given by \begin{equation}\label{pdf}p_t(x)dx:=P(\k B(t)\in dx)=\dfrac{e^{-\frac{1}{2t}\sum_{i,j}{\bf a}^{-1}_{i,j}x_ix_j}}{\l(2\pi t\r)^{d/2}\sqrt{det({\bf a})}}dx.\end{equation}
(Note: Since  $p(x)$ is not supported on a hyperplane in $\mathbb R^d$, ${\bf a}$ is positive definite.)

Define
\begin{equation}\label{cov_1}\begin{split}\sigma_1\left((s,\phi),(t,\psi)\right):=&\int_{\mathbb R^d}\int_{\mathbb R^d}\phi(y)\psi(z) p_{|t-s|}(z-y)dzdy\\
 &-\int_{\mathbb R^d}\int_{\mathbb R^d}\phi(y)\psi(z)p_{t+s}(z-y)dzdy\end{split}\end{equation}
 and 
 \begin{equation}\label{sigma_2}
 \begin{split}\sigma_2&\left((s,\phi),(t,\psi)\right):=\int_{\mathbb R^d}\int_{\mathbb R^d}\phi(y)\psi(z)p_{t+s}(z-y)dzdy\\
 &-\int_{\mathbb R^d}\int_{\mathbb R^d}\phi(y)\psi(z) p_{t}(z-y)dzdy-\int_{\mathbb R^d}\int_{\mathbb R^d}\phi(y)\psi(z) p_{s}(z-y)dzdy\\
&+\int_{\mathbb R^d}\phi(x)\psi(x)dx
\end{split}
\end{equation}
where $p_t(x)$ is the probability density function of $\k B(t)$ given in \eqref{pdf}.
\begin{theorem}\label{main_thm}
Under Assumptions \ref{configuration} - \ref{rw-moment}, as $n\to\infty$, $\xi_n(\cdot,\cdot)\to\xi(\cdot,\cdot)$ in distribution in  $D([0,T],\Sdual)$  where $\xi(\cdot,\cdot)$ is a mean zero, distribution-valued Gaussian process with covariance function
\begin{equation}\label{covariance-fn}
E\xi(s,\phi)\xi(t,\psi)=\rho_0\sigma_1\left((s,\phi),(t,\psi)\right)+v_0\sigma_2\left((s,\phi),(t,\psi)\right).
\end{equation}
\end{theorem}

\begin{theorem}\label{thm2}
Under Assumptions \ref{configuration} - \ref{rw-moment}, as $n\to\infty$, $\xi_n(\cdot,{\bf 1}_{B_M})$ converges in distribution in $D([0,T],\mathbb R)$ to a mean zero Gaussian process with covariance as in \eqref{covariance-fn} where $\phi=\psi={\bf 1}_{B_M}$.
\end{theorem}

{\it Example.} We calculate the covariance terms \eqref{cov_1} and \eqref{sigma_2} for a simple example. Take $\phi\equiv \psi:={\bf 1}_{B_M}=\prod_{i=1}^d{\bf 1}_{[-M,M]}$ and ${\bf a}=I_{d\times d}$. Let $\Phi_u(\cdot)$ denote the cumulative distribution function of a mean-zero real-valued Gaussian random variable (i.e. $\Phi_u(x):=\frac{1}{\sqrt{2\pi u}}\int_{-\infty}^xe^{-\frac{x^2}{2u}}$) and define
\[I_u=\l[2M\l\{\Phi_u(2M)-\Phi_u(-2M)\r\}+2\sqrt{\frac{u}{2\pi}}\l(e^{\frac{-2M^2}{u}}-1\r)\r]^d.\]   
Then
\[\sigma_1((s,\phi), (t,\psi))=I_{|t-s|}-I_{t+s},\] and
\[\sigma_2((s,\phi),(t,\psi))=I_{t+s}-I_t-I_s+(2M)^d.\]

%\begin{remark}
{\it Remark 1.} In the long-run, the current process is asymptotically a stationary Gaussian process. This can be observed by taking the long time asymptotic limit of the covariance terms \eqref{cov_1} and \eqref{sigma_2}.\\%\end{remark}

%\begin{remark}
{\it Remark 2.}
If we start with an initial configuration where $\eta_0(\cdot)$ are i.i.d. $Poisson(\lambda)$ random variables, the system of independent random walks is in equilibrium and we get $\eta_t(\cdot)$ are i.i.d.\ $Poisson(\lambda)$ for all $t\geq 0$. Under this invariant distribution of particles, the limiting current process in \cite{kumar} and \cite {Timo-IRW} are fractional Brownian motion with Hurst parameter $1/4$. In the present paper, under this invariant distribution of particles, $\rho_0=v_0=\lambda$ and
\begin{equation}\label{stationary-increments}\begin{split}
E[\xi(t,\phi)-&\xi(s,\psi)]^2=E[\xi(t,\phi)^2+\xi(s,\phi)^2-2\xi(t,\phi)\xi(s,\phi)]\\
&=\lambda \l[\int_{\mathbb R^d}\phi(x)\psi(x)dx-\int_{\mathbb R^d}\int_{\mathbb R^d}\phi(y)\psi(z)p_{t-s}(z-y)dzdy\r].
\end{split}\end{equation}
Thus, under the invariant distribution, $\xi(\cdot,\phi)$ has stationary increments, as $E(\xi(t,\phi)-\xi(s,\phi))^2$  is a function of $t-s$. The time-indexed process $\xi(\cdot,\phi)$ lacks the self-similarity property of one-paramter fBM. %; this is clearly observed by scaling time by a constant in the covariance terms  in \eqref{covariance-fn}.
 However,  the process $\xi(t,\phi)$  satisfies the following self-similarity property:
\[\xi(a t,\phi\circ\eta_a)\overset{d}{=}a^{d/2}\cdot\xi(t,\phi),\quad \text{for }a>0,\]
where $\eta_a(x)=\frac{x}{\sqrt{a}}, x\in\mathbb R^d$; this is evident from the form of the covariance terms \eqref{cov_1} and \eqref{sigma_2}.

%\end{remark}

The proof of both theorems involves two stages: first, showing convergence of finite dimensional distributions and second, proving tightness of the sequence of processes.
We use Mitoma's theorem \cite{Mitoma} to prove tightness of $\{\xi_n(\cdot,\cdot)\}\subset D([0,T],\Sdual)$.  According to Mitoma's theorem, it is sufficient to prove tightness for the sequence $\{\xi_n(\cdot, \phi)\}\subset D([0,T],\mathbb R)$ for each $\phi\in\S$. We use the tightness criteria in \cite{Durr} to prove tightness of the sequence $\{\xi_n(\cdot, \phi)\}\subset D([0,T],\mathbb R)$. A description of  Schwartz space $\S$, it's dual space $\Sdual$ and Mitoma's theorem can be found in \cite{Kallianpur}.

\def\A{\mathcal{A}}
\def\F{ \mathcal{F}}

\subsection{Stochastic integral representation}\label{sec:spde}
 The limiting current process in Theorem \ref{main_thm} has the same distribution as the solution of the stochastic partial differential equation given below in \eqref{spde}.
 
 Define operator $\A$ on $\S$  (the infinitesimal generator of $\k B(t)$) as
\begin{equation}
\A\phi(x):= \frac{1}{2}\sum_{i,j=1}^d\a_{i,j}\partial_{ij}\phi(x).
\end{equation}
Let $\{T_t, t\geq 0\}$ denote the semigroup associated to $\A$. Then \[T_t(\phi)(x)=\int_{\mathbb R^d}\phi(y)p_t(y-x)dy.\]
Define covariance functions
\begin{equation}
Q_1(\phi,\psi):=\int_{\mathbb R^d}\sum_{i,j=1}^d\a_{i,j}\partial_i\phi(x)\partial_j\psi(x)dx,
\end{equation}
and
\begin{equation}
Q_2(\phi,\psi):=\int_{\mathbb R^d}\A\phi(x)\A\psi(x)dx.
\end{equation}

Consider the following spde,
\begin{equation}
\label{spde}
dZ(t,\cdot)=\sqrt{\rho_0}dW^{\centerdot}_t+\A Z(t,\cdot)dt+\sqrt{v_0}\F(\cdot)dt; \qquad Z(0,\cdot)=0,
\end{equation}
where $W_t$ is a centered $\Sdual$-Wiener process with covariance $Q_1(\cdot.\cdot)$, $\{\F(\phi), \phi\in\S\}$ is a Gaussian random field on $\S$
 with covariance $Q_2(\cdot,\cdot)$ and $\F$ and $W_t$ are independent of each other.
We can write 
\begin{equation}\label{Weiner-process}
W_t^\phi= \int_{[0,t]\times \mathbb R^d}\nabla \phi(x)\cdot \k W_1(dtdx)
\end{equation}
where $W_1(t,x)$ is $d$-dimensional space-time white noise on $\mathbb R_+\times \mathbb R^d$.
\begin{equation}
\F(\phi)= \int_{\mathbb R^d}\A\phi(x)W_2(dx)
\end{equation}
where $W_2(x)$ is a one-dimensional white noise on $\mathbb R^d$ independent of the white noise $W_1$.
Observe that the stochastic differential equation \eqref{spde} has two independent sources of randomness and thus we get  the solution to \eqref{spde} to be the sum of two independent stochastic integrals
\begin{equation}
\label{soln}\begin{split}
Z(t,\phi)=&\sqrt{\rho_0}\int_{[0,t]\times\mathbb R^d}\nabla\l(T_{t-s}\phi(x)\r)\cdot\k W_1(dxds)\\&\qquad\qquad\qquad\qquad+\sqrt{v_0}\int_0^t\int_{\mathbb R^d}\A T_{t-s}\phi(x) W_2(dx)ds.\end{split}
\end{equation}
%It is easy to verify that  the solution to the spde \eqref{spde} is a distribution valued  Gaussian process with  the same  covariance as the limiting current process
Denote the two stochastic integrals  in \eqref{soln} as $Z_1(t,\phi)$ and $Z_2(t,\phi)$ so that
\[Z(t,\phi)=\sqrt{\rho_0}Z_1(t,\phi)+\sqrt{v_0}Z_2(t,\phi).\]
The first term $\sqrt{\rho_0}Z_1(t,\phi)$ represents space-time noise created by fluctuations in the random walks. The second term $\sqrt{v_0}Z_2(t,\phi)$ indicates the fluctuations in the initial configuration that get propagated as the system evolves. In the absence of fluctuations in the initial configuration, i.e. if we take the initial configuration to be deterministic ($v_0=0$), the second source of randomness in the spde \eqref{spde}, $\sqrt{v_0}\F(\cdot)dt$, disappears.  The spde \eqref{spde}  in this case then reduces to 
\[dZ(t,\cdot)=\sqrt{\rho_0}dW^{\centerdot}_t+\A Z(t,\cdot)dt, \]
where the solution 
\[Z(t,\phi)=\sqrt{\rho_0}\int_{[0,t]\times\mathbb R^d}\nabla\l(T_{t-s}\phi(x)\r)\cdot\k W_1(dxds)\] is a generalized Ornstein-Uhlenbeck process \cite{HS78}. 
 
For $t,s\in[0,T]$ and $\phi,\psi\in\S$,
\begin{equation*}\Cov (Z(t,\phi),Z(s,\psi))= \rho_0\Cov (Z_1(t,\phi),Z_1(s,\psi))+v_0\Cov (Z_2(t,\phi),Z_2(s,\psi)).\end{equation*}
Elementary calculations show that 
\begin{equation*}
\Cov (Z_1(t,\phi),Z_1(s,\psi))=\sigma_1\left((s,\phi),(t,\psi)\right)\end{equation*}
and 
\begin{equation*}
\Cov (Z_2(t,\phi),Z_2(s,\psi))=\sigma_2\left((s,\phi),(t,\psi)\right).
\end{equation*}
% We verify below that the covariance of \eqref{soln} is the same as the covariance of the limiting current process in Theorem \ref{main_thm}. 
We can  conclude that  the solution of the spde and the limiting current process in Theorem \ref{main_thm} are equal in distribution as they are both mean-zero, distribution-valued Gaussian processes with the same covariance.

\section{Convergence of finite dimensional distributions}
\label{sec:finite-dim}
 We begin the proof of Theorems \ref{main_thm} and \ref{thm2} with showing convergence of finite dimensional distributions.  \\
{\bf  Note:}  Henceforth, in this paper, `$c$' will denote constants that change from line to line in calculations.\\
Fix $N\in\mathbb N$ and choose $(t_i,\phi_i)\in [0,T]\times \mathcal S(\mathbb R^d)$ for $i=1,\hdots,N$ such that %$t_1\leq \hdots \leq t_N$ and 
$(t_i,\phi_i)\neq(t_j,\phi_j)$ for $i\neq j$.
 Let $(\theta_1,\hdots,\theta_N)\in\mathbb R^N$ be an arbitrary vector.
% Since $\{\eta_0(m):m\in\Zd\}$ are i.i.d., we get that $E\eta_t(m)=\rho_0$, $\forall t>0, m\in\Zd$ and hence $E\xi_n(t,\phi)=0$.

\begin{lemma}\label{finite-dim}
As $n\to\infty$, $\sum_{i=1}^N\theta_i\xi_n(t_i,\phi_i)$ converges to a mean-zero Gaussian random variable with variance \[\sigma^2=\sum_{i,j=1}^d\theta_i\theta_j\left(\rho_0\sigma_1\left((t_i,\phi_i),(t_j,\phi_j)\right)+v_0\sigma_2\left((t_i,\phi_i),(t_j,\phi_j)\right)\right).\]
\end{lemma}
\begin{proof}
 Define \[U_m(t,\phi):=n^{-d/4}\sum_{j=1}^{\eta_0(m)}\left[\phi\l(\Xt\r)-\phi\l(\tfrac{m}{\sqrt{n}}\r)\right].\]
 Let $W_m=\sum_{i=1}^N\theta_iU_m(t_i,\phi_i).$
 Denote $\bar{U}_m(t,\phi)=U_m(t,\phi)-EU_m(t,\phi)$ and $\bar{W}_m=W_m-EW_m.$
 Choose $r(n)$ so that $r(n)\to\infty$ slowly enough that 
 \begin{equation}\label{r(n)}
 (r(n))^d\cdot E[\eta_0(x)^21\{\eta_0(x)\geq n^{d/8}\}]\to 0\text{ as }n\to\infty.\end{equation}
 Since $\{\eta_0(m):m\in\Zd\}$ are i.i.d., a simple calculation shows that $E\eta_t(m)=\rho_0$, $\forall t>0, m\in\Zd$ and hence $E\xi_n(t,\phi)=0$.
Rewrite \[\sum_{i=1}^N\theta_i\xi_n(t_i,\phi_i)=\sum_{i=1}^N\theta_i\bar{\xi}_n(t_i,\phi_i)=\sum_{m\in\mathbb Z^d}\bar{W}_m.\] This sum can be split into the following two sums:
\begin{equation}
\label{two_sums}
\sum_{i=1}^N\theta_i\xi_n(t_i,\phi_i)=\sum_{|m|\leq r(n)\sqrt{n}}\bar{W}_m+\sum_{|m|>r(n)\sqrt{n}}\bar{W}_m.\end{equation}%S_1+S_2$
%where $S_1=\sum_{|m|\leq r(n)\sqrt{n}}\bar{W}_m$ and $S_2=\sum_{|m|>r(n)\sqrt{n}}\bar{W}_m$.
We can now apply the Lindeberg-Feller theorem to the sum $\sum_{|m|\leq r(n)\sqrt{n}}\bar{W}_m$ and show that it converges in distribution  to the mean zero Gaussian random variable with variance indicated in Lemma \ref{finite-dim}. The proof follows in the same vein as the proof of Lemma 2 in \cite{kumar} and we skip it.

We next  show that the second sum in \eqref{two_sums} converges to $0$ in $L^2$ as $n\to \infty$. The proof of this differs from the analogous step in \cite{kumar} and so we give the proof in detail. In \cite{kumar} we use large deviations to control contributions to the current from distant particles. This is not enough in the present situation and we need to appeal to the rapidly decreasing property of Schwartz functions to bound these contributions.

By Schwarz inequality,
\[E\l(\sum_{|m|>r(n)\sqrt{n}}\bar{W}_m\r)^2\leq ||\theta||^2\sum_{i=1}^N\sum_{|m|> r(n)\sqrt{n}}E\bar{U}_m(t_i,\phi_i)^2\]
where $||\theta||^2=\sum_{i=1}^N \theta_i^2$.
It suffices to show that for a fixed $t>0$ and $\phi\in \mathcal S(\mathbb R^d)$, \[\sum_{|m|> r(n)\sqrt{n}}E\bar{U}_m(t,\phi)^2\to 0\] as $n\to\infty$.

%As before, the `$c$' used below denotes constants that change from line to line.
\begin{align}\label{L2conv}
&\sum_{|m|> r(n)\sqrt{n}}E\bar{U}_m(t,\phi)^2\nonumber\\
&\leq cn^{-d/2}\sum_{|m|> r(n)\sqrt{n}}E\l[\sum_{j=1}^{\eta_0(m)}\phi(\Xt)-\rho_0E\phi(\Xtj)\r]^2\nonumber\\
&\qquad+cn^{-d/2}\sum_{|m|> r(n)\sqrt{n}}\phi^2(\tfrac{m}{\sqrt{n}})\nonumber\intertext{since $(a+b)^2\leq 2a^2+2b^2$ and $\eta_0(\cdot)$ have bounded moments by Assumption \ref{configuration} }
%&=n^{-d/2}\sum_{|m|> r(n)\sqrt{n}}[\rho_0Var(\phi(\Xtj))+v_0(E\phi(\Xtj))^2]\\
%&\qquad+cn^{-d/2}\sum_{|m|> r(n)\sqrt{n}}\phi^2(\tfrac{m}{\sqrt{n}})\\
%&\leq cn^{-d/2}\sum_{|m|> r(n)\sqrt{n}}\l[E\phi^2(\Xtj)+\l(E\phi(\Xtj)\r)^2+\phi^2(\tfrac{m}{\sqrt{n}})\r]\\
&\leq cn^{-d/2}\l[\sum_{|m|> r(n)\sqrt{n}}E\phi^2(\Xtj)+\sum_{|m|> r(n)\sqrt{n}}\phi^2(\tfrac{m}{\sqrt{n}})\r]\nonumber%\intertext{by H\"older's inequality}
\intertext{since $\eta_0(\cdot)$ have bounded second moments and are independent of the random walks} \nonumber
&=cn^{-d/2}\sum_{|m|> r(n)\sqrt{n}}\biggl[E\l(\phi^2(\Xtj){\bf 1}\{X_{m,1}(nt)\in B_{M(n)\sqrt{n}}+[n\vec{v}t]\}\r)\nonumber\\
&+E\l(\phi^2(\Xtj){\bf 1}\{X_{m,1}(nt)\notin B_{M(n)\sqrt{n}}+[n\vec{v}t]\}\r)+\phi^2(\tfrac{m}{\sqrt{n}})\biggr]
\intertext{where $M(n)$ is chosen so that $M(n)\to\infty$ as $n\to\infty$ and $M(n)=o(r(n))$.}\nonumber
&=cn^{-d/2}\l(I_1+I_2+I_3\r)
\end{align}
where $I_1,I_2,I_3$ denote the sums  of the three terms resp. in the expression above.
We show below that the scaled limits of each of these sums goes to $0$ as $n\to\infty$.%We show that $n^{-d/2}$ times  the sum over $|m|>r(n)\sqrt{n}$ of each of the three terms in parentheses  in the last line tend to $0$ as $n\to\infty$.

\begin{align*}
n^{-d/2}I_1&\leq cn^{-d/2}\sum_{|m|> r(n)\sqrt{n}}P(X(nt)\in B_{M(n)\sqrt{n}}+[n\vec{v}t]-m)\intertext{since $\phi$ is bounded}\\
%&=cn^{-d/2}\sum_{m\in B^c_{r(n)\sqrt{n}} }\sum_{i\in B_{M(n)\sqrt{n}}}P(X(nt)-[n\vec{v}t]=i+m)\\
&=cn^{-d/2}\sum_{i\in B_{M(n)\sqrt{n}}}P(X(nt)-[n\vec{v}t]\notin B_{r(n)\sqrt{n}}+i )\\
&\leq cn^{-d/2}\sum_{i\in B_{M(n)\sqrt{n}}}P(X(nt)-[n\vec{v}t]\notin B_{\frac{1}{2}r(n)\sqrt{n}} )
\intertext{as $\frac{1}{2}r(n)\sqrt{n}\leq r(n)\sqrt{n}-M(n)\sqrt{n}$ (recall that $M(n)=o(r(n)$)}
&=cn^{-d/2}\sum_{i\in B_{M(n)\sqrt{n}}}P(|X(nt)-[n\vec{v}t]|> \frac{1}{2}r(n)\sqrt{n})\\
&\leq cn^{-d/2}M(n)^dn^{d/2}\frac{E|X(nt)-[n\vec{v}t]|^r}{(r(n)\sqrt{n})^r}\intertext{by Markov inequality, $r\geq 0$}\\
&=cn^{r/2}n^{-r/2}M(n)^d(r(n))^{-r}\intertext{(as $E|X(nt)-[n\vec{v}t]|^{r}$ is O($n^{r/2}$))}
&\to 0 
\end{align*}
as $n\to\infty$ by taking $r>d$.

The following requires the rapidly decreasing property of Schwartz functions.
\begin{align*}
n^{-d/2}I_2
&\leq n^{-d/2}\sum_{m\in B^c_{r(n)\sqrt{n}} }\sum_{L\geq M(n)}E\Bigl(\phi^2(\Xtj)\\
&\qquad\qquad\qquad\qquad\qquad\qquad\times {\bf 1}\{X_{m,1}(nt)\in B^c_{L\sqrt{n}}\cap B_{(L+1)\sqrt{n}}+[n\vec{v}t]\}\Bigr)\\
&\leq n^{-d/2}\sum_{m\in B^c_{r(n)\sqrt{n}} }\sum_{L\geq M(n)}\frac{c_j}{L^{2j}}P(X(nt)-[n\vec{v}t]\in B^c_{L\sqrt{n}}\cap B_{(L+1)\sqrt{n}}-m)\intertext{since $\phi\in\mathcal S(\mathbb R^d)$, $|\phi(x)|\leq c_jx^{-j}$ for any $j\in\mathbb N$}
&\leq n^{-d/2}\sum_{m\in B^c_{r(n)\sqrt{n}} }\sum_{L\geq M(n)}\sum_{i\in B^c_{L\sqrt{n}}\cap B_{(L+1)\sqrt{n}}}\frac{c_j}{L^{2j}}P(X(nt)-[n\vec{v}t]=i -m)\\
&\leq n^{-d/2}\sum_{L\geq M(n)}\frac{c_j}{L^{2j}}\l|B^c_{L\sqrt{n}}\cap B_{(L+1)\sqrt{n}}\r| \intertext{by summing over $m$ first and then $i$}
& \leq n^{-d/2}\sum_{L\geq M(n)}\frac{c_j}{L^{2j}} (L\sqrt{n})^{d-1}\sqrt{n}= \sum_{L\geq M(n)}\frac{c_j}{L^{2j-d+1}} \to 0 
\end{align*}
as $n\to\infty$ by choosing $2j>d$ and since $M(n)\to \infty$ as $n\to\infty$.
And finally,
\begin{align*}
n^{-d/2}I_3\leq c\int_{|x|>r(n)} \phi^2(x)dx\to 0\end{align*}
as $n\to\infty$ by a Riemann sum argument and since $r(n)\to \infty$ as $n\to\infty$.

Thus the right hand side  of \eqref{L2conv} goes to $0$ as $n\to\infty$ and since
\[E\l(\sum_{|m|>r(n)\sqrt{n}}\bar{W}_m\r)^2\leq ||\theta||^2\sum_{i=1}^N\sum_{|m|> r(n)\sqrt{n}}E\bar{U}_m(t_i,\phi_i)^2,\]
we conclude that the second sum in \eqref{two_sums} goes to $0$ in $L^2$. This %, together with Lemma \ref{S}, 
proves Lemma \ref{finite-dim}.
\end{proof}

\section{Tightness and Completion of the Proof} 
\label{sec:tightness}
Choose $\alpha>0$ and 
\begin{equation}\label{beta}
\beta\geq d/4+d\alpha+1.
\end{equation}
% In the computations below, $c$ is used to denote positive constants that vary from line to line.
 Fix $\phi\in\S$.  We prove %We use the tightness criteria in \cite{Durr} to prove tightness of $\xi_n(\cdot,\phi)\in D([0,T],\mathbb R)$.  
 %DO I NEED TO STATE THE THEOREM FROM THAT PAPER????
\begin{prop}\label{process-tight}
 $\xi_n(\cdot,\phi)$ is tight in  $D([0,T],\mathbb R)$.
 \end{prop}
 To prove Proposition \ref{process-tight} we use the tightness criteria in \cite{Durr} .  %For the reader's convenience, the tightness criteria in \cite{Durr} is stated in Proposition \ref{Durr-thm} in the appendix. 
 In sections \ref{subsec:T1} and \ref{subsec:T2} below, we check the two tightness conditions of Proposition 5.7 in \cite{Durr} for the sequence of processes $\{\xi_n(\cdot,\phi)\}$.  %The tightness criteria comprises of two conditions which we verify below.  
 The first condition is a moment bound condition and the second involves the modulus of continuity.
\subsection{Verifying the first tightness condition}\label{subsec:T1}
Let $s,t\in[0,T]$ and  without loss of generality we assume $t>s$ below. Let $\bar{\xi}$ denote the centered current process.
Choose an integer \begin{equation}\label{r}r>2\beta>2.\end{equation} We show that
\begin{lemma}\label{tight-criterion-1}When $(t-s)\geq n^{-\beta}$,
\begin{equation}\label{tight-1}E\l|\bar{\xi}_n(t,\phi)-\bar{\xi}_n(s,\phi)\r|^{2r}\leq C_r (t-s)^\sigma \end{equation}where $\sigma>1$, and $C_r$ is a constant depending only on $r$.
\end{lemma}
 \begin{proof}
 Let 
\[ A_m:=\sum_{j=1}^{\eta_0(m)}\left(\phi(\Xt)-\phi(\Xs)\right).\]  We can write 
\[
\xi_n(t,\phi)-\xi_n(s,\phi)=n^{-d/4}\sum_{m\in\Zd}A_m.\]
Then
\begin{equation}\label{T1-bound}E|\bar{\xi}_n(t,\phi)-\bar{\xi}_n(s,\phi)|^{2r}=n^{-rd/2}E\left(\sum_{m\in\Zd}\bar{A}_m\right)^{2r}\end{equation}
where $\bar{A}_m=A_m-EA_m$.

To appropriately bound \eqref{T1-bound}, we first find the following moment bound.
% Estimates for A
\begin{lemma}
\label{moment_bound_A}
For any integer $1\leq k\leq 2r$, there exists a constant $C$ that depends on $r$, such that
\[\sm E\left|\bar{A}_m\right|^k\leq  C(n^{-1/2}+\sqrt{t-s}\ )n^{d/2}.\]
\end{lemma}
\begin{proof}
%{\allowdisplaybreaks
\begin{align}\label{A}
\sm E\left|\bar{A}_m\right|^k
&\leq 2^k\sm\l[E\l|A_m\right|^k\r]\nonumber\intertext{by applying $(a+b)^k\leq 2^k(|a|^k+|b|^k)$ and then H\"older's inequality}\nonumber
&\leq 2^k\sm\l[E(\eta_0(m))^kE\left|\pstj\right|^k\r]\nonumber\intertext{by Jensen's inequality and independence of $\eta_0(\cdot)$ from the random walks}
&\leq c2^{2r}\l[\sm E\left|\pstj\right|^k\r]
\end{align}
since $\eta_0(m)$ has bounded moments.

Define \[D\phi(x):=\l(\partial_1\phi(x),\hdots,\partial_d\phi(x)\r),\] where $\partial_i$ denotes the partial derivative with respect to the $i$th co-ordinate of $x$.
Let $\psi(x):=(1+|x|)^{-N}$ for some positive integer $N$.

\begin{align*}
\sm E&\left|\pstj\right|^k\\
&\leq c_N\sm \Biggl[E\biggl[\l(\psi(\Xsj)\r)^k\l|\Xtj-\Xsj\r|^k\\
&\qquad\qquad\times {\bf 1}\{|\Xtj|\geq |\Xsj|\}\biggr]\\
&+E\biggl[\l(\psi(\Xtj)\r)^k\l|\Xtj-\Xsj\r|^k\\
&\qquad\qquad\times{\bf 1}\{|\Xsj|> |\Xtj|\}\biggr] \Biggr]  
\end{align*}
by applying the mean value theorem and $|D\phi(x)|\leq c_N\psi(x)$.

We will treat one of the terms in the above inequality, the other being similar. 
\begin{align*}
c_N&\sm \Biggl[E\biggl[\l(\psi\l(\Xsj\r)\r)^k\l|\Xtj-\Xsj\r|^k\\
&\qquad\qquad\qquad\qquad\times{\bf 1}\l\{\l|\Xtj\r|\geq \l|\Xsj\r|\r\}\biggr]\\
&\leq c_N\sm\sum_{L\geq 0} E\Biggl[\frac{1}{(1+L)^{kN}} \l|\Xtj-\Xsj\r|^k\\
&\qquad\qquad\qquad\qquad\times{\bf 1}\l\{\Xsj\in B_{L}^c\cap B_{(L+1)}\r\}\Biggr]\\
&\leq c_N\sm\sum_{L\geq 0}\frac{1}{(1+L)^{kN}}E\Biggl[\l|\Xts\r|^k\\
&\qquad\qquad\qquad\qquad\times {\bf 1}\l\{X(ns)\in B_{L\sqrt{n}}^c\cap B_{(L+1)\sqrt{n}}+[n\vec{v}s]-m\r\}\Biggr] \\
&\leq c_N\sum_{L\geq 0}\frac{1}{(1+L)^{kN}}((L+1)\sqrt{n})^{d-1}\sqrt{n}E\l|\Xts\r|^k  \intertext{by summing over $m$ }
%&+c_N\sm\sum_{L\geq 0}\frac{1}{(1+L)^{kN}}E\Biggl[\l|\Xts\r|^k \\
%&\qquad\qquad\qquad\qquad\times{\bf 1}\l\{X(nt)\in B_{L\sqrt{n}}^c\cap B_{(L+1)\sqrt{n}}+[n\vec{v}t]-m\r\}\Biggr] \\
%&\leq c_N \sm E\l(\psi(\Xsj)\r)^kE\l|\Xtj-\Xsj\r|^k \\
%&+c\sm\sum_{L\geq 0} E\Biggl[\frac{1}{(1+L)^{kN}} \l|\Xtj-\Xsj\r|^k\\
%&\qquad\times {\bf 1}\l\{\Xtj\in B_{L}^c\cap B_{(L+1)}\r\}\Biggr]
%\intertext{using %the Lipschitz condition and 
%independence of $(X(nt)-X(ns))$ from $X(ns)$ for the first term}
%&\leq c \sm \sum_{L\geq 0}E\l[\frac{1}{(1+L)^{kN}}{\bf 1}\l\{\Xsj\in B_L^c\cap B_{(L+1)}\r\}\r]\\
%&\qquad\times E\l|\Xts\r|^k \\
%&+c\sm\sum_{L\geq 0}\frac{1}{(1+L)^{kN}}E\Biggl[\l|\Xts\r|^k \\
%&\qquad\qquad\qquad\qquad\times{\bf 1}\l\{X(nt)\in B_{L\sqrt{n}}^c\cap B_{(L+1)\sqrt{n}}+[n\vec{v}t]-m\r\}\Biggr] \\
%&\leq c \sm \sum_{L\geq 0}E\l[\frac{1}{(1+L)^{kN}}{\bf 1}\l\{\Xsj\in B_L^c\cap B_{(L+1)}\r\}\r]\\
%&\qquad\times E\l|\Xts\r|^k \\
%&+c\sum_{L\geq 0}\frac{1}{(1+L)^{kN}}(L\sqrt{n})^{d-1}\sqrt{n}E\l|\Xts\r|^k \intertext{summing over $m$ in the second term}
&\leq c_N\l[\sum_{L\geq 0}(1+L)^{-kN+d-1}\r]n^{d/2}[(t-s)^{k/2}+n^{-k/2}]\intertext{%summing over $m$ in the first term and
using central limit theorem}
%&\leq c n^{d/2}[(t-s)^{k/2}+n^{-k/2}]\intertext{ }
&\leq cn^{d/2}[\sqrt{t-s}+n^{-1/2}] \end{align*}by choosing $N$ large enough so that the sum in parenthesis is finite and by bounding $(t-s)^{k/2}\leq (2T)^{\frac{k-1}{2}}\sqrt{t-s}$ and $n^{-k/2}\leq n^{-1/2}$.
\end{proof}

\begin{lemma}
\label{A-bound}There exists a constant $C$ depending on $r$ such that
\begin{equation}
E\left(\sm\bar{A}_m\right)^{2r}\leq C\{1+(n^{-r/2}+(t-s)^{r/2}\ )n^{rd/2}\}%where const=r\cdot 2^{r}(2r)!.
\end{equation}
\end{lemma}
%\begin{proof}
The proof of the above lemma uses Lemma \ref{moment_bound_A} and is along the same lines as Lemma 8 in \cite{kumar}. (In fact Lemma 8 in \cite{kumar} is a special case where $r$ is $6$.) We omit the proof.  The reader is referred to Lemma 8 in \cite{kumar} for the idea of the proof.
%\end{proof}

By \eqref{T1-bound} and the above lemma we get, when $|t-s|\geq n^{-\beta}$,
\begin{align*}
E\l|\bar{\xi}_n(t,\phi)-\bar{\xi}_n(s,\phi)\r|^{2r}%\leq cn^{-rd/2}\Biggl[E\left(\sum_{m\in\Zd}\bar{A}_m\right)^{2r}\Biggl]\\
%&\leq cn^{-rd/2}\l[1+n^{rd/2}(t-s)^{r/2}+n^{r\frac{d-1}{2}}\r]\\
&\leq c\l[n^{-rd/2}+n^{-r/2}+(t-s)^{r/2}\r]\\
&\leq c\l[(n^{-\beta})^{\frac{rd}{2\beta}}+(n^{-\beta})^{\frac{r}{2\beta}}+(t-s)^{r/2}\r]\\
& \leq c (t-s)^\sigma
\end{align*}
 where $\sigma>1$ by our choice of $r>2\beta>2$ in \eqref{r} and where $c$ is a constant depending on $r$.
\end{proof}

%Choose $\alpha>0$ and $\beta\geq d/4+d\alpha+1$.
%We have the first Tightness condition which is:
%Choose $r>\max(2,2\beta)$, then
%\[E\l|\bar{\xi}_n(t,\phi)-\bar{\xi}_n(s,\phi)\r|^{2r}\leq c (t-s)^\sigma\]
%where $\sigma>1$.
\subsection{Verifying the second tightness condition}\label{subsec:T2}
The second tightness condition involves proving that
 \begin{equation}\label{modulus}\lim_{n\to\infty}P(w_{\xi_n}(n^{-\beta})>\epsilon)=0,\end{equation} where 
 \[w_{\xi_n}(n^{-\beta}):=\sup_{|t-s|<n^{-\beta}}|\xi_n(t,\phi)-\xi_n(s,\phi)|\]
is the modulus of continuity. 

Recall the definition of the step function $[\cdot]$ in \eqref{integer-part}.  For each $n$, we will divide $[0,T]$ into subintervals such that $[n\vec{v}t]$ is constant in each subinterval and each subinterval has length less than $n^{-\beta}$. 
We construct a family $\mathcal I_n$ of left-closed right-open subintervals by following the steps below.\\
%{\bf Algorithm}\\
Step 1: Set $a=0$.\\
Step 2: Let $b_1$ be the largest real number greater than $a$ such that $[n\v t]$ is constant as $t$ varies over interval $[a,b_1)$.\\
Step 3: Set $b=\min\{b_1,a+n^{-\beta}\}$.\\
Step 4: Include $[a,b)$ in $\mathcal I_n$.\\
Step 5: If $b>T$ then stop, otherwise set $a=b$ and repeat from Step 2.\\

Clearly $\mathcal I_n$ forms a minimal covering of $[0,T]$ that  satisfies  the following two conditions:\\
\begin{equation}
\label{subintervals}
\begin{split}
& 1)\   \forall [a,b)\in\mathcal I_n, |b-a|<n^{-\beta}\text{ and }[n\vec{v}t]\text{ is constant in }[a,b),\\
& 2)\   \text{the subintervals in }\mathcal I_n\text{ are disjoint.}
\end{split}
\end{equation}
%\item   \[[0,T]\subset \cup_{[a,b)\in \mathcal I_n}[a,b)\] 
%\item the intervals in $\mathcal I_n$ form a minimal covering of 
%Since $n^{1-\beta}|\v|<1$ for large $n$,  each co-ordinate in $[n\vec{v}t]$ can jump at most once as $t$ varies over an interval of length $n^{-\beta}$. Thus, in an interval of length $n^{-\beta}$, $[n\vec{v}t]$ can change values at most $d$ times, remaining constant the rest of the time. There are at most $[Tn^\beta+1]$ disjoint subintervals of length $n^{-\beta}$ in $[0,T+n^{-\beta})$.
%We thus get
A little thought  gives us
 \[|\mathcal I_n|\leq d [Tn^\beta+1]\text{ for large }n\] (i.e. number of such subintervals is at most $d [Tn^\beta+1]$ where $d$ is the dimension).

%We define the following modulus of continuity
% \[w_{\xi_n}(n^{-\beta}):=\sup_{|t-s|<n^{-\beta}}|\xi_n(t,\phi)-\xi_n(s,\phi)|.\]
% To show that \begin{equation}\label{modulus}\lim_{n\to\infty}P(w_{\xi_n}(n^{-\beta})>\epsilon)=0,\end{equation}
 To verify the second tightness condition, it is sufficient to show
  \begin{lemma} \label{tight-2}
For any $0<T<\infty$ and $\epsilon>0$, 
\begin{equation} 
\label{2}   
\lim_{n\to\infty}\sum_{[a,b)\in\mathcal I_n}P(\sup_{t\in[a,b)}|\xi_n(t,\phi)-\xi_n(a,\phi)|>\epsilon)=0.
\end{equation}   
\end{lemma} 
The reason that Lemma \ref{tight-2} is sufficient to prove \eqref{modulus} is the following. 
 For any $0\leq s\leq t\leq T$ with ${t-s}<n^{-\beta}$,  there exist $[a_1,b_1), [a_2,b_2)\in\mathcal I_n$ such that $s\in [a_1,b_1)$ and $t\in [a_2,b_2)$ ($[a_1,b_1)$ may be equal to $[a_2,b_2)$). %Choose $c$ such that $c-a_2=n^{-\beta}$, then b
 By the triangle inequality we get
\begin{equation}
\begin{split}
|\xi_n(t,\phi)-\xi_n(s,\phi)|&\leq |\xi_n(t,\phi)-\xi_n(a_2,\phi)|+|\xi_n(a_2,\phi)-\xi_n(a_2+n^{-\beta},\phi)|\\
&+|\xi_n(a_2+n^{-\beta},\phi)-\xi_n(a_1,\phi)|+|\xi_n(a_1,\phi)-\xi_n(s,\phi)|
\end{split}
\end{equation}  
%Note: Since $t-s<n^{-\beta}$, we have  $a_2-a_1<2n^{-\beta}$. For any interval $[a_2,b_2)\in\mathcal I_n$, there are no more than $2d$ intervals  $[a_1,b_1)\in I_n$ such that $a_2-a_1<2n^{-\beta}$. 
 By applying Lemma  \ref{tight-criterion-1} and $|\mathcal I_n|\leq d[Tn^\beta+1]$, we get
 \begin{align*}
 P(w_{\xi_n}(n^{-\beta})>\epsilon)
 &\leq 2\sum_{[a,b)\in\mathcal I_n}P(\sup_{t\in[a,b)}|\xi_n(t,\phi)-\xi_n(a,\phi)|>\epsilon/4)+cn^{\beta(1-\sigma)}
 \end{align*}
 %since $|\mathcal I_n|\leq d[Tn^\beta+1]$.
 As $n\to\infty$, the second term in the above inequality goes to $0$ as $\sigma>1$.

\begin{proof}[Proof of Lemma \ref{tight-2}] 
Let \begin{equation}\label{dag}\l(\dagger\r):= \sum_{[a,b)\in\mathcal I_n}P\l(\sup_{t\in[a,b)}\l|\xi_n(t,\phi)-\xi_n(a,\phi)\r|>\epsilon\r).\end{equation}
We want to show that \[\lim_{n\to\infty} \l(\dagger\r) =0.\]
We start by giving a general idea of the proof. The current process is written as the sum of two parts: contributions from particles starting inside the box $B_{n^{1/2+\alpha}}$ (recall $\alpha>0$ from beginning of this section) and contributions from  particles starting outside the box  $B_{n^{1/2+\alpha}}$. In the first case we bound the current by the number of jumps executed by particles initially within the box $B_{n^{1/2+\alpha}}$. In the second case we conclude that particles starting outside the box $B_{n^{1/2+\alpha}}$ can either travel a distance of order greater than $\sqrt{n}$ towards the origin, which would be a large deviation, or they would remain sufficiently far from the origin, in which case the value of $\phi$ at that distance from the origin would be small.\\

Choose $\gamma$ such that  \begin{equation}\label{gamma}
0<\gamma<\alpha.\end{equation}
We first find a bound on the expected number of particles  starting outside the box $B_{n^{1/2+\alpha}}$, which enter the box $B_{n^{1/2+\gamma}}$ at some point in time interval $[0,nT]$. 
\begin{lemma}
\label{lemman1}
 Define
\begin{equation}\label{n1}\begin{split}
N_1:=&\sum_{|m|> n^{1/2+\alpha}}\sum_{j=1}^{\eta_0(m)}{\bf 1} \{X_{m,j}(nt)\in B_{n^{1/2+\gamma}}+[n\vec{v}t]\\
&\text{for some }0\leq t\leq T\}.
\end{split}
\end{equation} Let $l\geq 1$. Then there exists a constant $c$ independent of $n$ such that  
$EN_1\leq cn^{d/2-l\alpha+d\gamma}$. %for any $l\geq 1$ and for large $n$, where $c$ is a constant independent of $n$.
\end{lemma}
\begin{proof}
Recall that $X(nt)$ denotes a random walk starting at the origin, with the same distribution as the evolution of the particles. Let $\bar{X}(nt):=X(nt)-[n\vec{v}t]$. 
\begin{align*}
EN_1&=\rho_0\sum_{|m|>n^{1/2+\alpha}} P(\bar{X}(nt)\in B_{n^{1/2+\gamma}}-m\text{ for some }t\in[0,T])\intertext{since $\eta_0$ is independent of the the random walks}
%&=\rho_0\sum_{m\in B^c_{n^{1/2+\alpha}}} \sum_{i\in B_{n^{1/2+\gamma}}}P(\bar{X}(nt)=i+m\text{ for some }t\in[0,T])\\
&=\rho_0\sum_{i\in B_{n^{1/2+\gamma}}}P(\bar{X}(nt)\notin B_{n^{1/2+\alpha}}+i \text{ for some }t\in[0,T])\\
&\leq \rho_0\sum_{i\in B_{n^{1/2+\gamma}}}P(\bar{X}(nt)\notin B_{\frac{1}{2}n^{1/2+\alpha}} \text{ for some }t\in[0,T])\intertext{for large enough $n$,  $\frac{1}{2}n^{1/2+\alpha}\leq n^{1/2+\alpha}-n^{1/2+\gamma}$, as $\gamma<\alpha$}
&=\rho_0\sum_{i\in B_{n^{1/2+\gamma}}}P(\sup_{0\leq t\leq T}|\bar{X}(nt)|> \frac{1}{2}n^{1/2+\alpha})\\
&\leq cn^{d/2}(n^\gamma)^d\frac{E|\bar{X}(nT)|^l}{(\frac{1}{2}n^{1/2+\alpha})^l}\intertext{by Doob's inequality for any $l\geq 1$}
&\leq cn^{d/2}(n^\gamma)^dn^{l/2}n^{-l\alpha-l/2}\intertext{as $E|\bar{X}(nT))|^{l}$ is O($n^{l/2}$)}
&\leq cn^{d/2-l\alpha+d\gamma}.
%\intertext{(choose $r>0$ large enough such that $d/2-r\alpha<0$)}\\
%&\to 0 \text{ as }n\to\infty
\end{align*}
\end{proof}

Let $[a,b)\in \mathcal I_n$. Define 
\[h_{m,j}^{[a,b)}(t):= \phi\l(\Xt\r)-\phi\l(\tfrac{X_{m,j}(na)-[n\v a]}{\sqrt{n}}\r).\]%, \text{   for $t\in [a,b)$.}\]
We can rewrite $(\dagger)$ in \eqref{dag} as
\begin{subequations}
\begin{align}
\l(\dagger\r)&=\sum_{[a,b)\in\mathcal I_n}P\l(\sup_{t\in[a,b)}\l|\sma\seta h_{m,j}^{[a,b)}(t)\r|>\frac{n^{d/4}\epsilon}{2}\r)\label{truncated}\\
&\qquad+\sum_{[a,b)\in\mathcal I_n}P\l(\sup_{t\in[a,b)}\l|\sum_{m\in B^c_{n^{1/2+\alpha}}}\seta h_{m,j}^{[a,b)}(t)\r|>\frac{n^{d/4}\epsilon}{2}\r)\label{N1bound}
%&\leq \sum_{[a,b)\in\mathcal I_n}P(\sup_{t\in[a,b)}\sma\seta {\bf 1}\{G_{m,j}^{[a,b)}\}>n^{d/4}\epsilon/2)+o(1)\intertext{by same argument used for bounding \eqref{N1bound}}
\end{align}
\end{subequations}
The two sums reflect the split in contributions to the current process from particles starting within $B_{n^{1/2+\alpha}}$ in \eqref{truncated}, versus particles starting outside $B_{n^{1/2+\alpha}}$ in \eqref{N1bound}.  

We first show that the second term \eqref{N1bound} goes to $0$ as $n\to\infty$.  To do this we split \eqref{N1bound} into two sums, the first containing contributions from particles that enter $B_{n^{1/2+\gamma}}$ at some time in $[0,nT]$ and the second containing contributions from particles that never enter $B_{n^{1/2+\gamma}}$.
\begin{subequations}
\begin{align}
&\sum_{[a,b)\in\mathcal I_n}P\l(\sup_{t\in[a,b)}\l|\sum_{m\in B^c_{n^{1/2+\alpha}}}\seta h_{m,j}^{[a,b)}(t)\r|>\frac{n^{d/4}\epsilon}{2}\r)\nonumber\\
&\leq  \sum_{[a,b)\in\mathcal I_n} P\biggl(\sum_{m\in B^c_{n^{1/2+\alpha}}}\seta c{\bf 1}\{X_{m,j}(nt)\in B_{n^{1/2+\gamma}}+[n\vec{v}t]\nonumber\\
&\qquad\qquad\quad\qquad\qquad\qquad\qquad\qquad \text{ for some }t\in[0,T] \}>\frac{n^{d/4}\epsilon}{4}\biggr)\label{inB}\\
&+\sum_{[a,b)\in\mathcal I_n}P\biggl(\sup_{t\in[a,b)}\sum_{m\in B^c_{n^{1/2+\alpha}}}\seta\l| h_{m,j}^{[a,b)}(t)\r|\nonumber\\
&\qquad\qquad\qquad\times{\bf 1}\l\{X_{m,j}(nt)\in B^c_{n^{1/2+\gamma}}+[n\vec{v}t], \  \forall t\in [a,b)\r\}>\frac{n^{d/4}\epsilon}{4}\biggr)\label{outB}
\end{align}
\end{subequations}
To get \eqref{inB}, we  bounded $h_{m,j}^{[a,b)}$ by some constant $c$ times the indicator function, since $\phi$ is a bounded function.

Now using Lemma \ref{lemman1}, we get
\begin{align*}
\eqref{inB}&\leq\sum_{[a,b)\in\mathcal I_n} P\l(cN_1\geq \frac{n^{d/4}\epsilon}{2}\r)\\
&\leq cn^\beta n^{-d/4}EN_1\intertext{by Markov inequality and since $|\mathcal I_n|\leq d[Tn^\beta+1]$}
&\leq c n^\beta n^{-d/4}n^{d/2-l\alpha+d\gamma}
\end{align*}
by Lemma \ref{lemman1}. Choose `$l$' large enough so that $\beta+d/4-l\alpha+d\gamma<0$. Then the right hand side $\to 0$ as $n\to\infty$.

We use the property that Schwartz functions  are rapidly decreasing to show \eqref{outB} goes to $0$.
\begin{align*}
\eqref{outB}&\leq \sum_{[a,b)\in\mathcal I_n}P\Bl(\sum_{m\in B^c_{n^{1/2+\alpha}}}\seta\\
&\sum_{L\geq n^\gamma}c_N(1+L)^{-N}{\bf 1}\l\{L\leq \inf_{t\in[a,b)}\l|\Xt\r|\leq (L+1)\r\}>\frac{n^{d/4}\epsilon}{4}\Br)\intertext{ since $|\phi(x)|\leq c_N(1+|x|)^{-N}$}
&\leq \frac{n^{-d/4}4}{\epsilon}\rho_0\sum_{[a,b)\in\mathcal I_n}\sum_{m\in B^c_{n^{1/2+\alpha}}}\sum_{L\geq n^\gamma}c_N(1+L)^{-N}\\
&\qquad\qquad\qquad\qquad\times P\l(\Xtj\in B_L^c\cap B_{(L+1)} \text{ for some }t\in[a,b)\r)\intertext{by Markov inequality and since $\eta_0$ is independent of the random walks}
&= \frac{n^{-d/4}4}{\epsilon}\rho_0\sum_{[a,b)\in\mathcal I_n}\sum_{m\in B^c_{n^{1/2+\alpha}}}\sum_{L\geq n^\gamma}c_N(1+L)^{-N}\\
&\qquad\qquad\times P\l(X(nt)-[n\v t]\in B_{L\sqrt{n}}^c\cap B_{(L+1)\sqrt{n}} -m \text{ for some }t\in[a,b)\r)\\
&\leq \frac{n^{-d/4}}{\epsilon}c_N\sum_{[a,b)\in\mathcal I_n}\sum_{L\geq n^\gamma}(1+L)^{-N}(L\sqrt{n})^{d-1}\sqrt{n}\intertext{by summing over $m$ and allowing $c_N$ to absorb all the other constants}
&\leq c_N \frac{n^{d/4+\beta}}{\epsilon}\sum_{L\geq n^\gamma}L^{-N+d-1}\intertext{since $|\mathcal I_n|\leq d[Tn^\beta+1]$}
%&\leq c_N\frac{n^{d/4+\beta}}{\epsilon}n^{-\gamma(N-d)}\intertext{by evaluating the sum using integrals}
&\leq c_Nn^{d/4+\beta-\gamma(N-d)}\to 0
\end{align*}
by choosing $N$ large enough so that $d/4+\beta-\gamma(N-d)<0$.

We finally turn to \eqref{truncated}. To prove \eqref{truncated} goes to $0$, we show that the number of particles  initially inside $B_{n^{1/2+\alpha}}$  that jump during a time interval of length $n^{1-\beta}$, is stochastically smaller than $O(n^{d/4})$.  

Recall, by definition of $\mathcal I_n$,  $[n\vec{v}t]$ is constant for $t\in [a,b)\in \mathcal I_n$.% Therefore $h_{m,j}^{[a,b)}(t)=0$ throughout time interval $[a,b)$ if $X_{m,j}(nt)$ does not jump during time interval $t\in [a,b)$.
Therefore for $t\in [a,b)$, \[|h_{m,j}^{[a,b)}(t)|\leq C {\bf 1}\{G_{m,j}^{[a,b)}\},\]
where $G_{m,j}^{[a,b)}:=\{X_{m,j}(nt)$ jumps during time interval $t\in [a,b)$\} and $C=2\sup_{x\in\mathbb R^d}\phi(x)$.
%Let \[I_1:=\sum_{[a,b)\in\mathcal I_n}P\l(\sup_{t\in[a,b)}\l|\sma\seta h_{m,j}^{[a,b)}(t)\r|>\frac{n^{d/4}\epsilon}{2}\r).\]
%Let $A\in\mathcal I_n, A\subset J_{k,n}$.
 Define \[\Lambda_n^{[a,b)}:=\log E\l[\exp\l\{\smG\r\}\r]\] and 
\[\tilde{\Lambda}^{[a,b)}_n:=\log E\l[\exp\l\{{\bf 1}\l\{G_{m,1}^{[a,b)}\r\}\r\}\r].\]
Let $\Pi(n^{1-\beta})$ be a  Poisson$(n^{1-\beta})$ random variable.
\begin{equation}\label{MGF-intermediate}
\begin{split}
\tilde{\Lambda}^{[a,b)}_n&=\log\l\{1+(e-1)P\l(X_{m,1}(nt)\text{ jumps during time interval }t\in [a,b)\r)\r\}\\
&\leq \log\l\{1+(e-1)P\l(\Pi(n^{1-\beta})\geq 1\r)\r\}\intertext{since $b-a\leq n^{-\beta}$}
&\leq (e-1)P\l(\Pi(n^{1-\beta})\geq 1\r)\\
&=c n^{1-\beta}
\end{split}
\end{equation}
by Markov inequality.
\begin{equation}\label{MGF}
\begin{split}
\Lambda_n^{[a,b)}&=\sma\log E\exp\l\{\eta_0(m)\tilde{\Lambda}_n^{[a,b)}\r\}\intertext{by independence of $\eta_0$ from the random walks}
%&\leq \sma \log E\l[1+\eta_0(m)cn^{1-\beta}+\frac{1}{2!}\l(cn^{1-\beta}\r)^2(\eta_0(m))^2+\hdots \r]\intertext{by \eqref{MGF-intermediate}}
&\leq \sma\log\l[1+\rho_0cn^{1-\beta}(e-1)+o(n^{1-\beta})\r]\intertext{by \eqref{MGF-intermediate} and  since $\eta_0(m)$ has exponential moments}
&\leq \sma cn^{1-\beta} %\intertext{where $c$ has absorbed all other constants}
\leq c n^{d/2+d\alpha+1-\beta}
\end{split}
\end{equation}
Putting all this together, we get
\begin{align*}
\eqref{truncated}&\leq \sum_{[a,b)\in\mathcal I_n}P\l(\sma\seta {\bf 1}\{G_{m,j}^{[a,b)}\}>\frac{n^{d/4}\epsilon}{2C}\r)
\intertext{ since $|h_{m,j}^{[a,b)}(t)|\leq C {\bf 1}\{G_{m,j}^{[a,b)}\}$ for $t\in [a,b)$}
&\leq  \sum_{[a,b)\in\mathcal I_n}\exp\l\{\frac{-n^{d/4}\epsilon}{2C}\r\}E\l[\exp\l\{\sma\seta {\bf 1}\{G_{m,j}^{[a,b)}\}\r\}\r]\intertext{by Markov inequality}
&\leq \sum_{[a,b)\in\mathcal I_n}\exp\l\{\frac{-n^{d/4}\epsilon}{2C}\r\}\exp\l\{\Lambda_n^{[a,b)}\r\}\\
&\leq d[Tn^\beta+1]\exp\l\{-cn^{d/4} \l(1-n^{d/4+d\alpha+1-\beta}\r)\r\}\intertext{ by applying \eqref{MGF} and $|\mathcal I_n|\leq d[Tn^\beta+1]$}%, $c$ absorbing all the constant coefficients }
&\to 0 
\end{align*}
as $n\to\infty$, since  $\beta\geq d/4+d\alpha+1$  by  definition \eqref{beta}.\\
This proves Lemma  \ref{tight-2} and thus verifies the second tightness condition \eqref{modulus}.
\end{proof}

\begin{proof}[Proof of Proposition \ref{process-tight}]
Lemma \ref{tight-criterion-1} and \eqref{modulus} satisfy the tightness  criteria in Proposition 5.7 of \cite{Durr},  thus we get $\xi_n(\cdot,\phi)$ is tight in $D([0,T],\mathbb R)$.
\end{proof}

\begin{proof}[Proof of Theorem \ref{main_thm}]
We invoke  Theorem 4.1 in \cite{Mitoma} %Theorem 2.5.1 in \cite{Kallianpur}, 
which states that Proposition \ref{process-tight} is sufficient to prove that the sequence $\{\xi_n(\cdot,\cdot)\}$ is tight in $D([0,T],\Sdual)$. This, together with Lemma \ref{finite-dim} proves the theorem. 
\end{proof}

\begin{proof}[Proof of Theorem \ref{thm2}]
%The proof of this follows in much the same fashion as the proof for Theorem \ref{main_thm}. %The proof of Theorem \ref{main_thm} can be easily modified to prove Theorem \ref{thm2}. 
The proof of this theorem follows almost verbatim from that of Theorem \ref{main_thm}. A few places where the proof differs from that of Theorem \ref{main_thm} are highlighted below.

When proving tightness, the proof of Lemma \ref{moment_bound_A} is different.
From equation \eqref{A} onwards in Lemma \ref{moment_bound_A}, the proof differs as follows.
\begin{equation*}\begin{split}
\sm E\left|\bar{A}_m\right|^k&\leq c2^{2r}\l[\sm E\left|\pstj\right|^k\r]\\
%&= c2^k\l[\sm E\left|{\bf 1}_{C_m}-{\bf 1}_{D_m}\right|^k\r]\intertext{since $\phi={\bf 1}_{B_M}$}
&\leq c2^{2r}\sm\l[P(C_m)+P(D_m)\r]
\end{split}\end{equation*}
where $C_m=\{X_{m,1}(nt)\in B_{M\sqrt{n}}+[n\vec{v}t], X_{m,1}(ns)\notin B_{M\sqrt{n}}+[n\vec{v}s]\}$\\
and $D_m=\{X_{m,1}(nt)\notin B_{M\sqrt{n}}+[n\vec{v}t],X_{m,1}(ns)\in B_{M\sqrt{n}}+[n\vec{v}s]\}$.

\begin{align*}
&\sum_{m\in\mathbb Z^d}P(C_m)=P(X(nt)\in B_{M\sqrt{n}}+[n\vec{v}t]-m, X(ns)\notin B_{M\sqrt{n}}+[n\vec{v}s]-m)\\
%&=\sum_{m\in\mathbb Z^d}\sum_{j\notin B_{M\sqrt{n}}}P(X(nt)\in B_{M\sqrt{n}}+[n\vec{v}t]-m,X(ns)=j+[n\vec{v}s]-m)\\
&=\sum_{m\in\mathbb Z^d}\sum_{j\notin B_{M\sqrt{n}}}P(X(nt)-X(ns)\in B_{M\sqrt{n}}+[n\vec{v}t]-[n\vec{v}s]-j\ \vert\\
&\qquad X(ns)=j+[n\vec{v}s]-m) \times P(X(ns)=j+[n\vec{v}s]-m)\\
&=\sum_{j\notin B_{M\sqrt{n}}}P(X(n(t-s))\in B_{M\sqrt{n}}+[n\vec{v}t]-[n\vec{v}s]-j)\\
&=\sum_{j\notin B_{M\sqrt{n}}}P(\bar{X}_n(t,s)\in B_{M\sqrt{n}}+j)\\
\intertext{where $\bar{X}_n(t,s)=X(n(t-s))-[n\vec{v}t]+[n\vec{v}s]$}
&=\sum_{k\in\mathbb Z^d}P(\bar{X}_n(t,s)=k)\sum_{j\notin B_{M{\sqrt{n}}}}{\bf 1}\{k\in B_{M\sqrt{n}}+j\}\\
&\leq \sum_{l=0}^{M\sqrt{n}}\sum_{k\in \mathbb Z^d:|k|=l}P(\bar{X}_n(t,s)=k)n^{\frac{d-1}{2}}|k|+\sum_{|k|>M\sqrt{n}}P(\bar{X}_n(t,s)=k)cn^{d/2}\\
&\leq n^{\frac{d-1}{2}}E|\bar{X}_n(t,s)|+cn^{d/2}P(\bar{X}_n(t,s)>M\sqrt{n})\\
&\leq c\bigl[ n^{\frac{d-1}{2}}(\sqrt{n(t-s)}+1)+n^{d/2}\frac{E|\bar{X}_n(t,s)|^2}{M^2n} \bigr]\intertext{since $E\l|X(n(t-s))-n(t-s)\v \r|$ is $O\l(\sqrt{n(t-s)}\ \r)$}
&\leq c\bigl[ n^{\frac{d-1}{2}}(\sqrt{n(t-s)}+1)+n^{d/2}\frac{n(t-s)+1}{M^2n} \bigr]
%&\leq c\bigl[ n^{d/2}\sqrt{t-s}+n^{\frac{d-1}{2}}+n^{d/2}(t-s)+n^{d/2-1} \bigr]\\
\leq cn^{d/2}[\sqrt{t-s}+n^{-1/2}].
%c\bigl[ n^{d/2}\sqrt{t-s}+n^{\frac{d-1}{2}}+n^{d/2}(t-s) \bigr]
\end{align*}

Similarly,
\begin{align*}
\sum_{m\in\mathbb Z^d}P(D_m)\leq cn^{d/2}[\sqrt{t-s}+n^{-1/2}].%c\bigl[ n^{d/2}\sqrt{t-s}+n^{\frac{d-1}{2}}+n^{d/2}(t-s) \bigr]
\end{align*}

The proof of \eqref{modulus}  holds for indicator functions without any modifications. 
\end{proof}

{\bf Acknowledgments}  This  paper is part of my Ph.D. thesis at the University of Wisconsin-Madison. I would like to thank my advisor, Prof.\ Timo Sepp\"al\"ainen, for his guidance. I would also like to thank Prof.\ Tom Kurtz for pointing out the connection to stochastic partial differential equations and for several helpful discussions.

\bibliographystyle{plain}
%\bibliography{bib-ddim}

\end{document}